\documentclass[a4paper,twoside]{article}
\usepackage{amssymb,amscd,amsfonts,amsbsy,nicefrac}
\usepackage{amsmath,amsthm,amssymb,amsfonts}
\usepackage{mathrsfs}
\usepackage[pdftex]{graphicx,color}
\usepackage{mathrsfs}
\usepackage[british]{babel}
\usepackage[a4paper,tmargin=3cm,bmargin=3cm,lmargin=3cm,rmargin=3cm]{geometry}

\newtheorem{thm}{Theorem}[section]
\newtheorem{lem}[thm]{Lemma}

\newtheorem{cor}[thm]{Corollary}

\newtheorem{defi}[thm]{Definition}

\def\R{{\mathbb R}}

\def\C{{\mathbb C}}

\def\q{{\frac{1}{|B|}}}
\def\2q{{\frac{2}{|B|}}}
\newcommand{\N}{\mathbb{N}}

\newcommand{\Z}{\mathbb{Z}_{+}}

\newcommand{\esssup}{\mbox{ess sup}}

\newcommand{\supp}{\mbox{supp}}

\renewcommand{\leq}{\leqslant}
\renewcommand{\geq}{\geqslant}

\begin{document}

\title{Multilinear pseudodifferential operators beyond Calder\'on-Zygmund theory}

%\address{School of Mathematics and the Maxwell Institute of Mathematical Sciences, University of Edinburgh, James Clerk Maxwell Building, The King's Buildings, Mayfield Road, Edinburgh, EH9 3JZ, United Kingdom}
%\email{Nicholas.Michalowski@ed.ac.uk}
%\thanks{}

\author{Nicholas Michalowski\footnotemark[2]{,} David J.\ Rule\footnotemark[3]{,} \& Wolfgang Staubach\footnotemark[3]{}\,\,\footnotemark[1]{}}
\footnotetext[1]{The third author is partially supported by the EPSRC First Grant Scheme, reference number EP/H051368/1.}
%\address{Department of Mathematics and the Maxwell Institute of Mathematical Sciences, Heriot-Watt University, Colin Maclaurin Building, Edinburgh, EH14 4AS, United Kingdom}
%\email{rule@uchicago.edu}
%\thanks{}
\date{}

%\author[W.\ Staubach]{}
%\address{Department of Mathematics and the Maxwell Institute of Mathematical Sciences, Heriot-Watt University, Colin Maclaurin Building, Edinburgh, EH14 4AS, United Kingdom}
%\email{W.Staubach@hw.ac.uk}
%\thanks{}
%\subjclass[2000]{42B20, 42B25, 35S05, 47G30}
%\keywords{Weighted norm inequality, Multilinear Pseudodifferential operator, Multilinear Pseudo-pseudodifferential operator}

\maketitle

\begin{abstract}
We consider two types of multilinear pseudodifferential operators. First, we prove the boundedness of multilinear pseudodifferential operators with symbols which are only measurable in the spatial variables in weighted Lebesgue spaces. These results generalise earlier work of the present authors concerning linear pseudo-pseudodifferential operators. Secondly, we investigate the boundedness of bilinear pseudodifferential operators with symbols in the H\"ormander $S^{m}_{\rho, \delta}$ classes. These results are new in the case $\rho < 1$, that is, outwith the scope of multilinear Calder\'on-Zygmund theory.
\end{abstract}

\section{Introduction}
The study of multilinear pseudodifferential operators goes back to the pioneering works of R.~Coifman and Y.~Meyer, \cite{CM1}, \cite{CM3}, \cite{CM4} and \cite{CM}. Since then, there has been a large amount of work on various generalisations of their results, as well as studies of bilinear operators with symbols satisfying different conditions to those in the standard bilinear Coifman-Meyer classes. The literature in this area of research is vast and any brief summary of it here would not do the authors justice. Therefore we confine ourselves to mention only those works with a direct connection to the present paper.

R.~Coifman and Y.~Meyer, in \cite{CM4} and \cite{CM}, proved the boundedness from $L^{p_1}\times L^{p_2} \times \dots \times L^{p_N}$ to $L^r$ of multilinear pseudodifferential operators with symbols in the class $S^0_{1,0}(n,N)$ (see Definition \ref{bilinear multiplier defn} below) for $1 < p_i < \infty$ and $r > 1$ with $1/p_1 + 1/p_2 + \dots + 1/p_n = 1/r$. In the seminal paper \cite{GT}, L.~Grafakos and R.~Torres systematically developed the theory of multilinear Calder\'on-Zygmund operators. They proved a multilinear $T(1)$-Theorem which they applied to generalise the result above to $r > 1/N$. As a further application, they demonstrated the boundedness in Lebesgue spaces of multilinear pseudodifferential operators which, together with each of the adjoint operators, belonged to $OPS^0_{1,1}(n,N)$ (see Definition \ref{def4}).

However, in \cite{BT2}, A.~B\'enyi and R.~Torres showed that there exist symbols in $S^0_{1,1}(n,2)$ that do not give rise to bilinear operators which are bounded from $L^{p_1} \times L^{p_2}$ to $L^r$ for $1 \leq p_1,p_2,r < \infty$ such that $1/p_1 + 1/p_2 = 1/r$. In particular, there is no analogue of the Calder\'on-Vaillancourt Theorem in the bilinear setting. Moreover, the class of operators $OPS^0_{1,1}(n,2)$ is not closed under transposition. In contrast, \cite{BT1} demonstrates that $OPS^0_{1,0}(n,2)$ is closed under transposition.

Recently, in \cite{BMNT}, A.~B\'enyi, D.~Maldonado, V.~Naibo and R.~Torres proved that $OPS^m_{\rho,\delta}(n,2)$ is closed under transposition for $0 \leq \delta \leq \rho \leq 1$ and $\delta < 1$. In particular, given an operator in $OPS^0_{1,\delta}(n,2)$, its adjoints are also in $OPS^0_{1,\delta}(n,2)$. Since $S^0_{1,\delta}(n,2) \subset S^0_{1,1}(n,2)$, it follows that symbols in $S^0_{1,\delta}(n,2)$ give rise to bounded operators, by applying the result of \cite{GT} quoted above.

In summary, we see that $OPS^0_{\rho,\delta}(n,2)$ are bounded on appropriate Lebesgue spaces when $\rho = 1$ (that is, the Calder\'on-Zygmund case), but in general they fail to be bounded when $\rho = 0$. The purpose of this paper is to address the following question, which is of interest for $\rho$ in-between these values, `Given $\rho \leq 1$, what $m = m(\rho) \leq 0$ is sufficient to ensure that symbols in $S^m_{\rho,\delta}(n,N)$ give rise to bounded operators?' This question is in the spirit of questions asked in \cite{BMNT}.

We will study this question for two different symbol classes. First, in Section \ref{main1}, we will consider a larger symbol class which does not require any differentiability in the spatial variable at all. That is, we study the multilinear symbol class $L^\infty S^m_{\rho}(n,N)$ (see Definition \ref{def2}) which, in particular, contains $S^m_{\rho,\delta}(n,N)$ for any $\delta$. Our main result in this context is Theorem \ref{multi_one}, which generalises a result obtained by the present authors in \cite{MRS1} regarding the linear case. The study of such symbol classes originates in \cite{KS}, where C.~Kenig and the third author studied linear operators. In the context of multilinear operators, results regarding mildly regular bilinear operators have been proved previously. In particular, D.~Maldonado and V.~Naibo established in \cite{MN} boundedness properties of bilinear pseudodifferential operators on products of weighted Lebesgue, Hardy, and amalgam spaces. The regularity they require in the spatial variables is only that of Dini-type. Section \ref{mixed norms} deals with linear operators on mixed-norm Lebesgue spaces, and is a corollary to the proof of Theorem \ref{multi_one}.

The second topic we will study is the bilinear symbol class $S^m_{\rho,\delta}(n,2)$. In Section \ref{LpSmrho and applications} we adapt methods used to study symbols in $L^\infty S^m_{\rho}(n,N)$ to weaken the requirement on $m$ necessary to prove boundedness on Lebesgue spaces of operators in $OPS^m_{\rho,\delta}(n,2)$ for $\delta \leq \rho$. This is formulated as Theorem \ref{smooth pseudo bilinear}. In Section \ref{outwith}, although we cannot show boundedness for general operators arising from symbols in $S^0_{\rho,\delta}(n,2)$, we can prove boundedness on a suitable subclass. This is stated as Theorem \ref{modulationtype}, which is a result of the same flavour as that proved by F.~Bernicot and S.~Shrivastava in \cite{BS} regarding a subclass of $OPS^0_{0,0}(1,2)$, albeit proved by more straight-forward methods. A related result regarding $OPS^0_{0,0}(n,2)$ was also proved in \cite{BT2}.

We begin the main body of the paper with Section \ref{prelim} where we set out some definitions, fix some notation and recall some well-known results that we will use later.

\section{Definitions, Notation and Preliminaries} \label{prelim}

We study the following type of \emph{multilinear pseudodifferential operator}. Given a function $a \colon \R^n \times \R^{nN} \to \C$ we define the $N$-linear operator $T_a$ to act on $N$ functions $u_1, \dots, u_N$ belonging to the Schwartz class $\mathscr{S}$ as
\begin{equation} \label{def1}
T_a (u_1,\dots,u_N)(x):=\int_{\R^{nN}} a(x,\Xi) \prod_{j=1}^N \widehat{u}_j(\xi_j) e^{ix\cdot\xi_j} d\,\Xi.
\end{equation}
Here $x,\xi_1,\dots,\xi_N$ are all variables in $\R^n$, $\Xi = (\xi_1,\dots,\xi_N) \in \R^{nN}$ and $\widehat{u} \colon \R^n \to \C$ denotes the Fourier transform
\begin{equation*}
\widehat{u}(\xi) = \int_{\R^n} u(y) e^{-i\xi\cdot y} dy
\end{equation*}
of $u\in \mathscr{S}$. We refer to the function $a$, which has $(N + 1)n$ variables, as the \emph{symbol} of the operator $T_a$.

We set $\Xi:=(\xi_1 ,\dots, \, \xi_N )$ with $\xi_j \in \mathbb{R}^n$, and define $|\Xi|^2:= \sum_{j=1}^{N} |\xi_j|^{2},$ where $|\xi_j|$ denotes the standard Euclidean norm of $\xi_j \in\R^n .$ Also, here and in the sequel we shall use $\Z$ to denote the set of nonnegative integers.\\
We will use a standard Littlewood-Paley partition of unity $\{\varphi_{k}\}_{k\geq 0}$ in $\R^{nN},$ by letting $\varphi_{0} \colon \R^{nN} \to \R$ be a smooth radial function
which is equal to one on the unit ball centred at the origin and supported on its concentric double. Setting
$\varphi(\Xi) = \varphi_0(\Xi) - \varphi_0(2\,\Xi)$ and $\varphi_k(\Xi) = \varphi(2^{-k}\,\Xi)$ for $k \geq 1,$ we have
\begin{equation} \label{lp}
\varphi_{0}(\Xi)+\sum_{k=1}^{\infty}\varphi_{k}(\Xi) =1 \quad\text{for all }
\Xi \in \R^{nN},
\end{equation}
and $\supp(\varphi_{k})\subset\{\Xi \mid 2^{k-1} \leq |\Xi| \leq 2^k\}$ for
$k\geq 1$. One also has, for all multi-indices $\alpha \in \Z^{Nn}$ and $N \geq 0$,
\begin{equation*}\label{eq2.2}
|\partial ^{\alpha}_{\Xi}\varphi _{0}(\Xi)| \leq
c_{\alpha, N} \langle \Xi\rangle ^{-N},
\end{equation*}
where $\langle \Xi\rangle:=(1+|\Xi|^{2})^{\frac{1}{2}}$, and
\begin{equation}\label{eq2.3}
 |\partial ^{\alpha}_{\Xi}\varphi _{k}(\Xi)| \leq c_{\alpha}
2^{-k|\alpha|}\quad \text{for some } c_{\alpha }>0
\text{ and all } k\geq 1.
\end{equation}

\begin{defi} \label{def2}
Given $N\in \N$, $m \in \R$ and $\rho \leq 1$ the symbol $a \colon \R^n \times \R^{Nn} \to \C$ is said to belong to $L^\infty S^m_{\rho}(n,N)$ when for each multi-index $\alpha \in \Z^{Nn}$ there exists a constant $C_\alpha$ such that
\begin{equation*} \label{ks1}
\esssup_x |\partial_\Xi^\alpha a(x,\Xi)| \leq C_\alpha \langle \Xi\rangle^{m-\rho |
\alpha|}.
\end{equation*}
\end{defi}
In the case $N=1$ we also use the notation $L^{\infty} S^{m}_{\rho}$ for the class of symbols of the linear pseudo-pseudodifferential operators, see \cite{KS}.

\begin{defi} \label{def4}
Given a class of symbols $X$, operators which arise from elements in $X$ are denoted by $OPX$. That is, we say $T\in OPX$ when there exists a symbol $a\in X$ such that $T = T_a$, as defined in \eqref{def1}. Consequently, for $a \in L^{\infty} S^{m}_{\rho} (n,N)$ we say $T_a \in OPL^{\infty} S^{m}_{\rho}(n,N)$.
\end{defi}

For a non-negative function $\mu$, which we refer to as a \emph{weight}, we define $L^p_\mu = L^p_\mu(\R^n)$ to be the closure of $u \in \mathscr{S}$ in the norm
\[
\left(\int_{\R^n} |u(x)|^p \mu(x)dx\right)^\frac{1}{p}.
\]
When $\mu \equiv 1$ we write simply $L^p = L^p(\R^n)$ to mean $L^p_1$ and $L^{p}_{\mathrm{loc}}$ is the class of functions which belong to $L^p_\mu$ for each $\mu$ which is the characteristic function of a compact set.

We wish to study the boundedness from $L^{q_1}_{w_1} \times \dots \times L^{q_N}_{w_N}$ to $L^{r}_{\mu}$ of the operator $T_a$, initially defined for Schwartz functions $u_1,\dots,u_N$ via \eqref{def1}, for particular exponents $q_1,\dots,q_N,r$ and weights $w_1,\dots,w_N,\mu$. Although the integral in \eqref{def1} may not be absolutely convergent for $u_1, \dots, u_N$ which do not decay sufficiently rapidly, if we can prove bounds on the operator norm which depend only on $q_1,\dots,q_N,r,w_1,\dots,w_N,\mu,n,N$ and $a$, then it is a straight-forward exercise to show that $T_a$ has a unique extension to $L^{q_1}_{w_1} \times \dots \times L^{q_N}_{w_N}$ which agrees with \eqref{def1} for $u_1,\dots,u_N \in \mathscr{S}$. This is the sense in which we will refer to the boundedness of $T_a$.

Given $u\in L^{p}_{\mathrm{loc}}$, the $L^p$ maximal function $M_p(u)$ is defined by
\begin{equation}
M_p(u)(x) = \sup_{B\ni x} \left(\q \int_{B} \vert u(y)\vert^{p} \, dy\right)^{\frac{1}{p}}
\end{equation}
where the supremum is taken over balls $B$ in $\R^{n}$ containing $x$. Clearly then, the Hardy-Littlewood maximal
function is given by
\[
M(u) := M_{1}(u).
\]
An immediate consequence of H\"older's inequality is that $M(u)(x)\leq M_{p}(u)(x)$ for $p\geq 1$.
We shall use the notation
\[
u_{B}:= \q \int_{B} \vert u(y)\vert \, dy
\]
for the average of the function $u$ over $B$.
One can then define the class of Muckenhoupt $A_p$ weights as follows.
\begin{defi} \label{weights}
Let $w\in L^{1}_{\mathrm{loc}}$ be a positive function. One says that $w\in A_1$ if there exists a constant $C>0$ such that
\begin{equation*}
M w (x)\leq C w(x),\,\,\, \text{for almost all} \,\,\, x \in \R^{n}.
\end{equation*}
One says that $w\in A_p$ for $p\in(1,\infty)$ if
\begin{equation*}
\sup_{B\, \textrm{balls in}\,\, \R^{n}}\,w_{B}(w^{-\frac{1}{p-1}})_{B}^{p-1}<\infty.
\end{equation*}
The $A_p$ constants of a weight $w\in A_p$ are defined by
\begin{equation*}
[w]_{A_1}:=   \sup_{B\, \textrm{balls in}\,\, \R^{n}}\,w_{B}\Vert w^{-1}\Vert_{L^{\infty}(B)},
\end{equation*}
and
\begin{equation*}%\label{ap constant}
[w]_{A_p}:=  \sup_{B\, \textrm{balls in}\,\, \R^{n}}\,w_{B}(w^{-\frac{1}{p-1}})_{B}^{p-1}.
\end{equation*}
\end{defi}

The following results are well-known and can be found in, for example, \cite{S}.

%\begin{thm} \label{open}
%Suppose $p > 1$ and $w \in A_p$. There exists an exponent $q < p$, which depends only on $p$ and $[w]_{A_p}$, such that $w\in A_q$.
%There exists $\varepsilon > 0$, which depends only on $p$ and $[w]_{A_p}$, such that $w^{1+\varepsilon} \in A_p$.
%\end{thm}

\begin{thm} \label{maxweight}
For $1 < q < \infty$, the Hardy-Littlewood maximal operator is bounded on $L^q_w$ if and only if $w \in A_q$. Consequently, for $1 \leq p <
\infty$, $M_p$ is bounded on
$L^q_w$ if and only if $w \in A_{q/p}$
\end{thm}

\begin{thm} \label{convolve}
Suppose that $\phi \colon \R^n \to \R$ is integrable non-increasing and radial. Then, for $u \in L^1$, we have
\[
\int \phi(y)u(x-y) \, dy \leq \|\phi\|_{L^1} M(u)(x)
\]
for all $x \in \R^n$.
\end{thm}

We will need the following multilinear version of the Hausdorff-Young theorem due to A.~Benedek and R.~Panzone~\cite{BP}.

\begin{thm} \label{hausdorff-young}
Suppose that $1 \leq p_N \leq p_{N-1} \leq \dots \leq p_1 \leq 2$ and
\begin{equation*}
K(x_1,\dots,x_n) = \int\dots\int a(\xi_1,\dots,\xi_n) \prod_{j=1}^N e^{i x_j\cdot\xi_j} d\xi_1 \dots d\xi_N.
\end{equation*}
Then
\begin{equation*}
\begin{aligned}
& \Big\{ \int\dots\Big\{\int \Big\{ \int |K(x_1,\dots,x_N)|^{p_1'} dx_1\Big\}^{\frac{p_2'}{p_1'}} dx_2 \Big\}^\frac{p_3'}{p_2'} \dots dx_N \Big\}^\frac{1}{p_N'} \\
& \leq \Big\{ \int\dots\Big\{\int \Big\{ \int |a(\xi_1,\dots,\xi_N)|^{p_1} d\xi_1\Big\}^{\frac{p_2}{p_1}} d\xi_2 \Big\}^\frac{p_3}{p_2} \dots d\xi_N \Big\}^\frac{1}{p_N}.
\end{aligned}
\end{equation*}
\end{thm}

As is common practice, we will denote constants which can be determined by known parameters in a given situation, but whose
value is not crucial to the problem at hand, by $C$. Such parameters in this paper would be, for example, $m$, $\rho$, $p$, $n$, $[w]_{A_p}$, and the constants $C_\alpha$ in Definition \ref{def2}. The value of $C$ may differ
from line to line, but in each instance could be estimated if necessary. We also write $a\lesssim b$ as shorthand for $a\leq Cb$.

\section{A pointwise estimate for operators in $OPL^\infty S^m_\rho (n,N)$ and the weighted boundedness of multilinear operators} \label{main1}

The following lemma will be useful in obtaining pointwise estimates for the kernel of operators in $OPL^\infty S^m_\rho (n,N).$ For $Z=(z_1 , \dots , z_N)$ with $z_j \in \mathbb{R}^n$ and $\Xi= (\xi_1 , \dots, \xi_N)$ with $\xi_j \in \mathbb{R}^n ,$ we define $$\langle Z, \Xi\rangle:= \sum_{j=1}^{N} z_j \cdot \xi_j.$$

\begin{lem}\label{zlemma}
Let $a\in L^\infty S^{m}_\rho(n,N)$ with $m\in \R$ and $\rho \in (0,1]$. Given any $Z=(z_1 , \dots , z_N)\in \R^{nN}$ such that the set $S:=\{ j \in [1, N] \,|\, |z_j|\geq 1\} \neq \emptyset$, one has
\begin{equation*}\label{elem estim}
\Big|\int_{\mathbb{R}^{nN}} a(x,\Xi)\, e^{i\langle Z, \Xi\rangle}\, d\,\Xi\Big| \lesssim \prod_{j\in S} |z_j|^{-N_j},
\end{equation*}
for all $x\in\R^{n}$ and $\min_{j \in S} N_j \geq 0.$
\end{lem}

\begin{proof}
Setting $a_{k} (x, \Xi ):= a(x, \Xi ) \varphi_{k}(\Xi)$, and using the
definition of $L^{\infty}S_{\rho}^{m}(n,N)$, inequality \eqref{eq2.3} and
the Leibniz rule we see that
\begin{equation}\label{a_k bounds}
|\partial ^{\alpha}_{\Xi}a_{k}(x,\Xi)| \leq c_{\alpha} 2^{k(m-\rho| \alpha|) }, \,\, \text{for some} \,\, c_{\alpha }>0 \,\, \text{and} \,\, k= 1, 2,\dots
\end{equation}
and
\begin{equation}\label{a_0 bounds}
|\partial ^{\alpha}_{\Xi}a_{0}(x,\Xi)| \leq c_{\alpha, M} \langle \Xi\rangle^{-M}, \,\, \text{for some}
\,\, c_{\alpha,M}>0 \,\, \text{and} \,\, M \geq 0.
\end{equation}
where in \eqref{a_k bounds} we have also used the assumption $\rho \leq 1$.
We claim that
\begin{equation}\label{kernel a_0 estim}
\Big|\int_{\mathbb{R}^{nN}} a_{0}(x,\Xi)\, e^{i\langle Z, \Xi\rangle}\, d\,\Xi\Big| \lesssim \prod_{j=1}^{N}\langle z_j\rangle ^{-L},
\end{equation}
for all $L\geq 0$. Integrating by parts and using \eqref{a_0 bounds} with $M>nN$ yields
\begin{equation*}
\Big|\int_{\mathbb{R}^{nN}} \Big\{\prod_{j=1}^{N} z_{j}^{\alpha_j}\Big\}\,a_{0}(x,\Xi)\, e^{i\langle Z, \Xi\rangle}\, d\,\Xi\Big|
= \Big|\int_{\mathbb{R}^{nN}}\Big\{ \prod_{j=1}^{N}\partial ^{\alpha_ j}_{\xi_j}\,a_{0}(x,\Xi)\Big\}\, e^{i\langle Z, \Xi\rangle}\, d\,\Xi\Big|
\lesssim \int_{\mathbb{R}^{nN}}\langle \Xi\rangle^{-M}\, d\,\Xi\lesssim 1,
\end{equation*}
Now summing both sides of the above estimate over all $\alpha_j$ with $\sum_{j=1}^{N} |\alpha_j|\leq L$ and using the straightforward inequality $\prod_{j=1}^{N}\langle z_j\rangle ^{L} \lesssim \sum_{|(\alpha_1,\dots,\alpha_N)|\leq L}\prod_{j=1}^{N} z_{j}^{\alpha_j}$, we obtain \eqref{kernel a_0 estim}.
For the integrals containing $a_k$, integration by parts and \eqref{a_k bounds} yield
\begin{equation*}
\Big|\Big\{\prod_{j=1}^{M} z_{j}^{\alpha_j}\Big\} \int_{\mathbb{R}^{nN}} a_{k}(x,\Xi)\, e^{i\langle Z, \Xi\rangle}\, d\,\Xi\Big|
= \Big|\int_{\mathbb{R}^{nN}} \Big\{\prod_{j=1}^{M}\partial ^{\alpha_ j}_{\xi_j}\, a_{k}(x,\Xi)\Big\} e^{i\langle Z, \Xi\rangle}\, d\,\Xi\Big| \lesssim 2^{k(nN+m-\rho \sum_{j=1}^{M}|\alpha_j|)}.
\end{equation*}
Therefore, if $\sum_{j=1}^{M}|\alpha_j| >\frac{nN+m}{\rho}$ then $\Big|\sum_{k=1}^{\infty}\int_{\mathbb{R}^{nN}} a_{k}(x,\Xi)\, e^{i\langle Z, \Xi\rangle}\, d\,\Xi\Big|\lesssim \prod_{j=1}^{M}  |z_{j}| ^{-|\alpha_j |}.$ From this and the definition of the set $S$, by taking $\sum_{j=1}^{M}|\alpha_j| >\frac{nN+m}{\rho}$ and $|\alpha_{j}|\geq N_j$, it follows that
\begin{equation}\label{kernel a_k estim}
\Big|\sum_{k=1}^{\infty}\int_{\mathbb{R}^{nN}} a_{k}(x,\Xi)\, e^{i\langle Z, \Xi\rangle}\, d\,\Xi\Big|\lesssim \prod_{j\in S}  |z_{j}| ^{-N_j},
\end{equation}
for $N_j \geq 0$ and $Z \in \R^{nN}$ such that $S \neq \emptyset$. The estimate for $\int_{\mathbb{R}^{nN}} a(x,\Xi)\, e^{i\langle Z, \Xi\rangle}\, d\,\Xi$ follows by combining the estimates \eqref{kernel a_0 estim} and \eqref{kernel a_k estim}. This proves the lemma.
\end{proof}
As an immediate corollary we have the following kernel estimates
\begin{cor} \label{zlemma2}
Let $K(x,Y):=\int_{\mathbb{R}^{nN}}
a(x,\Xi) \prod_{j=1}^N e^{i(x-y_j) \cdot \xi_j}\, d\,\Xi$ and suppose $x$ and $Y$ are such that $S := \{j \in [1, N] \,|\, |x-y_j|\geq 1\} \neq \emptyset$.  Then one has
\begin{equation}\label{kernel1}
|K(x,Y)| \lesssim \prod_{j\in S}|x-y_j|^{-N_j}\quad \text{when} \quad
\min_{j \in S} N_j \geq 0
\end{equation}
provided either $\rho>0$ and $m \in \R$, or $\rho=0$ and $m<-nN$.
\end{cor}
\begin{proof}
When $\rho > 0$, this follows from Lemma \ref{zlemma} by setting $z_j= x-y_j$. An examination of the proof of Lemma \ref{zlemma} reveals that it can be easily modified for the case $\rho = 0$ provided $m<-nN$.
\end{proof}
The following theorem is the main result of this section.

\begin{thm} \label{multi_one}
Fix $p_j \in [1,2]$ for $j=1,\dots,N$ and let $a \in L^\infty S^m_\rho(n,N)$ with $0 \leq \rho \leq 1$ and $m < (\rho - 1)\sum_{j=1}^N\frac{n}{p_j}$. Then there
exists a constant $C$, depending only on $n$, $p_j$, $m$, $\rho$ and a finite number of the constants $C_\alpha$ in Definition
$\ref{def2}$, such that
\begin{equation} \label{pointws}
|T_a(u_1,\dots,u_N)(x)|\leq C \prod_{j=1}^N M_{p_j}(u_j)(x),
\end{equation}
for all $x \in \R^n$. Consequently, for $p_j < q_j \leq \infty$ and $r > 0$ such that $\frac{1}{r} = \sum_{j=1}^N \frac{1}{q_j}$, $T_a$ is a bounded operator from $L^{q_1}_{w_1} \times \dots \times L^{q_N}_{w_N}$ to $L^{r}_\mu$ whenever
\[
\mbox{$w_j \in A_{q_j/p_j}$ if $q_j < \infty$ or $w_j \equiv 1$ if $q_j=\infty$ for $j=1,\dots,N$,}
\]
and $\mu = \prod_{j=1}^N w_j^{r/q_j}$.
\end{thm}

\begin{proof}
The boundedness follows immediately from the pointwise estimate \eqref{pointws} by Theorem \ref{maxweight}.

To prove \eqref{pointws} we use the Littlewood-Paley partition of unity as in \eqref{lp}, we decompose the symbol as
\begin{equation*}
a(x,\Xi)= a_{0}(x,\Xi)+\sum_{k=1}^{\infty}a_{k}(x,\Xi)
\end{equation*}
with $a_k(x,\Xi)= a(x,\Xi)\varphi_{k}(\Xi)$, $k\geq 0$.

First we consider the operator $T_{a_0}$. We have
\begin{equation*}%\label{eqa0}
\begin{aligned}
T_{a_{0}}(u_1,\dots,u_N)(x) & =\int_{\mathbb{R}^{nN}}\int_{\mathbb{R}^{nN}} a_{0}(x,\Xi)\, \prod_{j=1}^N e^{i (x-y_j) \cdot \xi_j}\, u_j(y_j)\, dY\,d\,\Xi \\
& = \int_{\mathbb{R}^{nN}} K_0 (x,Y) \prod_{j=1}^N \, u_j(x-y_j) \,dY,
\end{aligned}
\end{equation*}
with
\begin{equation*}
K_0 (x,Y)=\int_{\mathbb{R}^{nN}} a_{0}(x,\Xi)\, e^{i \langle Y, \Xi \rangle}\,d\,\Xi.
\end{equation*}
Now estimate \eqref{kernel a_0 estim} yields
\begin{equation*}
|K_0(x,Y)|\lesssim \prod_{j=1}^{N}\langle y_j\rangle ^{-L},
\end{equation*}
for each $L>0$ and hence for $L>n$. Therefore Theorem \ref{convolve} implies
\begin{equation} \label{pointwiseakestim2}
|T_{a_{0}}(u)(x)| \lesssim \prod_{j=1}^N \int_{\mathbb{R}^n} \langle y_j\rangle ^{-L} |u_j(x-y_j)|\, dy_j \lesssim \prod_{j=1}^N M(u_j)(x) \lesssim \prod_{j=1}^N M_{p_j}(u_j)(x),
\end{equation}
for any $1\leq p_j$.

Now let us analyse
$T_{a_{k}}(u_1,\dots,u_N)(x)=\int a_{k}(x,\Xi)
\prod_{j=1}^N\hat{u}_j(\xi)e^{i x \cdot \xi_j}d\,\Xi$ for $k\geq 1$. We note, just as before, that
$T_{a_k}(u_1,\dots,u_N)(x)$ can be written as
\begin{equation*}
T_{a_{k}}(u_1,\dots,u_N)(x)=\int_{\mathbb{R}^{nN}} K_{k}(x,Y)\prod_{j=1}^N u_j(x-y_j)dY
\end{equation*}
with
\begin{equation*}
K_{k}(x,Y)=\int_{\mathbb{R}^{nN}} a_{k}(x,\Xi) \prod_{j=1}^N e^{i y_j \cdot \xi_j}\,d\,\Xi.
\end{equation*}
One observes that
\begin{equation*}
 |T_{a_{k}}(u_1,\dots,u_N)(x)| =
\Big|\int_{\mathbb{R}^{nN}} K_{k}(x,Y)\prod_{j=1}^N u_j(x-y_j)\,dY\Big| =
\Big|\int_{\mathbb{R}^{nN}} K_{k}(x,Y)\prod_{j=1}^N \sigma_k^j(y_j)\frac{u_j(x-y_j)}{\sigma_k^j(y_j)}\,dY\Big|,
\end{equation*}
where the weight functions $\sigma_k^j$ will be chosen
momentarily. Therefore, H\"older's inequality yields
\begin{equation}\label{eq2.14}
\begin{aligned}
& |T_{a_{k}}(u_1,\dots,u_N)(x)| \\
& \leq
\Big\{ \int\dots\Big\{\int \Big\{ \int |K_{k}(x,Y)|^{p_1'} \prod_{j=1}^N|\sigma_{k}^j(y_j)|^{p_1'} dy_1\Big\}^{\frac{p_2'}{p_1'}} dy_2 \Big\}^\frac{p_3'}{p_2'} \dots dy_N \Big\}^\frac{1}{p_N'} \times \\
& \prod_{j=1}^N  \Big\{ \int \frac{|u_j(x-y_j)|^{p_j}}{|\sigma_k^j(y_j)|^{p_j}}dy_j\Big\}^\frac{1}{p_j},
\end{aligned}
\end{equation}
where $\frac{1}{p_j} +\frac{1}{p_j'}=1$. Now for an $s_j>n/p_j$, we
define $\sigma_k^j$ by
\begin{equation*}
\sigma_{k}^j(y)=\begin{cases}
2^{\frac{-k\rho n}{p_j}}, & |y| \leq 2^{-k\rho}; \\
2^{-k\rho(\frac{n}{p_j}-s_j)}|y|^{s_j}, & |y| >
2^{-k\rho}.
\end{cases}
\end{equation*}
We now wish to estimate
\begin{equation*}
\Big\{ \int\dots\Big\{\int \Big\{ \int |K_{k}(x,Y)|^{p_1'} \prod_{j=1}^N |\sigma_{k}^j(y_j)|^{p_1'} dy_1\Big\}^{\frac{p_2'}{p_1'}} dy_2 \Big\}^\frac{p_3'}{p_2'} \dots dy_N \Big\}^\frac{1}{p_N'}
\end{equation*}
by splitting each $y_j$-integral as integration over $|y_j| \leq 2^{-k\rho}$
and $|y_j| > 2^{-k\rho}$. Considering an arbitrary case of the $2^N$ possibilities, we can estimate this portion of the integral by
\begin{equation*}
\sum_{\alpha_1,\dots,\alpha_N}\Big\{ \int\dots\Big\{\int \Big\{ \int |K_{k}(x,Y)|^{p_1'} \Big|\prod_{j=1}^N 2^{-k\rho(\frac{n}{p_j}-|\alpha_j|)}y_j^{\alpha_j}\Big|^{p_1'} dy_1\Big\}^{\frac{p_2'}{p_1'}} dy_2 \Big\}^\frac{p_3'}{p_2'} \dots dy_N \Big\}^\frac{1}{p_N'},
\end{equation*}
where the sum is taken over multi-indices $\alpha_j$ (each with $n$ components) such that $|\alpha_j| = 0$ if $|y_j| \leq 2^{-k\rho}$ and $|\alpha_j| = s_j$ if $|y_j| > 2^{-k\rho}$. Without loss of generality we may assume $1 \leq p_n \leq p_{n-1} \leq \dots \leq p_1 \leq 2$, so by Theorem \ref{hausdorff-young} and the estimate \eqref{a_k bounds}, this in turn is majorised by
\begin{equation*}
\begin{aligned}
& \sum_{\alpha_1,\dots,\alpha_N} \Big\{ \int\dots\Big\{\int \Big\{ \int |\prod_{j=1}^N 2^{-k\rho(\frac{n}{p_j}-|\alpha_j|)} \partial_{\xi_j}^{\alpha_j} a_{k}(x,\Xi)|^{p_1} d\xi_1\Big\}^{\frac{p_2}{p_1}} d\xi_2 \Big\}^\frac{p_3}{p_2} \dots d\xi_N \Big\}^\frac{1}{p_N} \\
& \lesssim 2^{k\left(m - (\rho - 1) \sum_{j=1}^N \frac{n}{p_j}\right)},
\end{aligned}
\end{equation*}
Furthermore, once again using Theorem \ref{convolve}, we have
\begin{equation*}
\Big\{\int \frac{|u_j(x-y_j)|^{p_j}}{|\sigma_{k}^j(y_j)|^{p_j}} dy_j \Big\}^\frac{1}{p_j} \lesssim M_{p_j} (u_j)(x)
\end{equation*}
with a constant that only depends on the dimension $n$. Combining these facts with \eqref{eq2.14} yields
\begin{equation}\label{pointwiseakestim}
 |T_{a_{k}}(u_1,\dots,u_N)(x)| \lesssim 2^{k\left(m- (\rho -1)\sum_{j=1}^N\frac{n}{p_j}\right)}\prod_{j=1}^N M_{p_j}(u_j)(x)
\end{equation}

Summing in $k$, and using \eqref{pointwiseakestim2} and \eqref{pointwiseakestim}, we obtain
\begin{equation*}
\begin{split}
|T_a (u_1,\dots,u_N)(x)|& \lesssim \sum_{k=0}^{\infty}|T_{a_{k}}(u_1,\dots,u_N)(x)| \\
& \lesssim \sum_{k=1}^{\infty}2^{k\left(m- (\rho -1)\sum_{j=1}^N\frac{n}{p_j}\right)} \prod_{j=1}^N M_{p_j}(u_j)(x).
\end{split}
\end{equation*}
We observe that the series above converges if $m< (\rho - 1)\sum_{j=1}^N\frac{n}{p_j}$. This proves \eqref{pointws} and, with it, the theorem.
\end{proof}

We remark in passing that the case $p_1 = p_2 = \dots = p_N$ Theorem \ref{multi_one} follows from its linear predecessor, that is Theorem 3.3 in \cite{MRS1}, with $\R^{nN}$ replacing $\R^n$.

\section{Boundedness of linear operators in mixed norm spaces}\label{mixed norms}

In this section we show that a modification of the proof of Theorem \ref{multi_one} yields a mixed norm boundedness result for a class of linear pseudodifferential operators. Let $X = (x_1,\dots,x_N) \in \R^{nN}$ with $x_j \in \R^n$ for $j = 1,\dots,N.$ To a symbol $a\in L^\infty S^m_\rho(nN,1),$ we associate a linear operator $T_{a},$ a-priori defined on functions $u$ in $\mathscr{S}(\R^{nN}),$ given by
\begin{equation*} %\label{def6}
T_a (u)(X):=\int_{\R^{nN}} a(X,\Xi) \widehat{u}(\Xi) e^{i\langle X,\Xi\rangle} d\Xi,
\end{equation*}
where $\widehat{u}$ is the Fourier transform of $u$ in $\R^{nN}$:
\begin{equation} \label{ftN}
\widehat{u}(\Xi) = \int_{\R^{nN}} u(X) e^{i\langle X,\Xi\rangle} dX.
\end{equation}
We define the space $L^{p_N}L^{p_{N-1}}\dots L^{p_1}$ to be the mixed norm space which is the closure of $u \in \mathscr{S}(\R^{nN})$ in the norm
\[
\|u\|_{L^{p_N}L^{p_{N-1}}\dots L^{p_1}} = \Big\{ \int\dots\Big\{\int \Big\{ \int |u(X)|^{p_1} dx_1\Big\}^{\frac{p_2}{p_1}} dx_2 \Big\}^\frac{p_3}{p_2} \dots dx_N \Big\}^\frac{1}{p_N}.
\]
We will also need the following notation. For a function $u \colon \R^{nN} \to \C$ we define $M_p^{(j)}(u)$ to be the $L^p$ maximal function acting only in the $x_j$-variables. That is,
\[
M_p^{(j)}(u)(X) := \sup_{B\ni x_j} \left(\q \int_{B} \vert u(x_1,\dots,x_{j-1},y,x_{j+1},\dots,x_N)\vert^p \, dy\right)^\frac{1}{p},
\]
for $j=1,\dots,N$, where the supremum is taken over balls $B$ in $\R^n$ containing $x_j$.
\begin{thm} \label{multi_two}
Fix $p_j \in [1,2]$ for $j=1,\dots,N$ and let $a \in L^\infty S^m_\rho(nN,1)$ with $0 \leq \rho \leq 1$ and $m < (\rho - 1)\sum_{j=1}^N\frac{n}{p_j}$. Then there
exists a constant $C$, depending only on $n$, $p_j$, $m$, $\rho$ and a finite number of the constants $C_\alpha$ in Definition
$\ref{def2}$, such that
\begin{equation} \label{ptmn}
|T_a(u)(X)|\leq C M_{p_N}^{(N)}(\dots M_{p_2}^{(2)}(M_{p_1}^{(1)}(u)))(X),
\end{equation}
for all $X \in \R^{nN}$. Consequently, $T_a$ is a bounded operator from $L^{q_1}L^{q_2} \dots L^{q_N}$ to $L^{q_N}L^{q_{N-1}} \dots L^{q_1}$ whenever $q_j > p_j$ $($$j=1,\dots,N$$)$ and $1 \leq q_N \leq q_{N-1} \leq \dots \leq q_1 \leq 2$.
\end{thm}
\begin{proof}
We repeat the prove of Theorem \ref{multi_one}, but with the linear operator above and observe that \eqref{eq2.14} is replaced by
\begin{equation}\label{eq4.2}
\begin{aligned}
& |T_{a_{k}}(u)(X)| \\
& \leq
\Big\{ \int\dots\Big\{\int \Big\{ \int |K_{k}(X,Y)|^{p_1'} \prod_{j=1}^N|\sigma_{k}^j(y_j)|^{p_1'} dy_1\Big\}^{\frac{p_2'}{p_1'}} dy_2 \Big\}^\frac{p_3'}{p_2'} \dots dy_N \Big\}^\frac{1}{p_N'} \times \\
& \Big\{\int \dots \Big\{\int \Big\{\int |u(X-Y)|^{p_1} \frac{dy_1}{|\sigma_k^1(y_1)|^{p_1}}\Big\}^\frac{p_2}{p_1} \frac{dy_2}{|\sigma_k^2(y_2)|^{p_2}}\Big\}^\frac{p_3}{p_2} \dots \frac{dy_N}{|\sigma_k^N(y_N)|^{p_N}}\Big\}^\frac{1}{p_N}.
\end{aligned}
\end{equation}
We control the first factor on the right-hand side of \eqref{eq4.2} as before. To control the second factor in \eqref{eq4.2} we can use Theorem \ref{convolve} to show that it is majorised by
\[
M_{p_N}^{(N)}(\dots M_{p_2}^{(2)}(M_{p_1}^{(1)}(u)))(X).
\]
By combining these estimates we obtain \eqref{ptmn}. Using the boundedness of the maximal function and Minkowski's inequality repeatedly, we obtain the boundedness of $T_a$ on the mixed norm space defined above.
\end{proof}

\section{The class $L^p_\mu S^m_\rho$ and an application to the boundedness of smooth bilinear operators} \label{LpSmrho and applications}

We now consider linear pseudodifferential operators acting on functions on $\R^n$.

\begin{defi} \label{def5}
Let $\mu$ be a weight $($that is, a non-negative function$)$, and $1 \leq p \leq \infty$, $m \in \R$ and $\rho \leq 1$ be parameters. A symbol $a \colon \R^n \times \R^{n} \to \C$ belongs to the class $L^p_\mu S^m_{\rho}$ if for each multi-index $\alpha \in \Z^{n}$ there exists a constant $C_\alpha$ such that
\begin{equation*} \label{symr}
\|\partial_\xi^\alpha a(\cdot,\xi)\|_{L^p_\mu(\R^n)} \leq C_\alpha \langle \xi\rangle^{m-\rho |
\alpha|}.
\end{equation*}
When the weight $\mu\equiv 1$ then we use the notation $L^pS^m_\rho$ for $L^p_1 S^m_\rho$.
\end{defi}

\begin{thm} \label{Lrsym}
Suppose $r \in [1,\infty)$, $q \in (1,\infty]$ and $p \in [2,\infty)$ with conjugate $p'$ $($for which $1/p + 1/p' = 1$$)$ satisfy the relation $1/r = 1/q + 1/p$. Suppose further that $\rho \leq 1$ and $m < n(\rho-1)/p'$. Let $a \in L^p_{\mu}S^m_\rho$ and $\mu$ and $w$ be weights with $w \in A_{q/p'}$. Then there exists a constant $C$, depending only on $n$, $m$, $\rho$, $p$, $q$, $[w]_{A_{q/p'}}$ and a finite number of $C_\alpha$ from $\mathrm{Definition}$ $\ref{def5}$, such that
\begin{enumerate}

\item[$(i)$]  \label{weightedLp} if $r \neq 1$ and $q \neq \infty$, then
\[
\|T_a(u)\|_{L^r_{\nu}(\R^n)} \leq C\|u\|_{L^q_w(\R^n)},
\]
where $\nu = \mu^{r/p} w^{r/q}$; and

\item[$(ii)$] \label{unweightedLp} if $r = 1$ or $q = \infty$, then
\[
\|T_a(u)\|_{L^r_{\nu}(\R^n)} \leq C\|u\|_{L^q(\R^n)},
\]
where $\nu = \mu^{r/p}$.
\end{enumerate}
\end{thm}

\begin{proof}
First, we define
\begin{equation*}
\sigma_{k}(y)=\begin{cases}
2^{\frac{-k\rho n}{p'}}, & |y| \leq 2^{-k\rho}; \\
2^{-k\rho(\frac{n}{p'}-\ell)}|y|^{\ell}, & |y| >
2^{-k\rho}.
\end{cases}
\end{equation*}
It is then easy to check that
\[
\left(\int |\sigma_k(y)|^{-p'}dy \right)^\frac{1}{p'} \lesssim 1
\]
provided $\ell$ is sufficiently large.
We consider the Littlewood-Paley pieces $T_{a_k}$ of the operator $T_{a}$, where $a_k(x,\xi) = a(x,\xi)\varphi_k(\xi)$. Using H\"older's inequality, the Hausdorff-Young inequality and Theorem \ref{convolve}, we compute
\begin{equation} \label{5.1}
\begin{aligned}
|T_{a_k}(u)(x)| & = \left| \int K_k(x,y)u(x-y)dy\right| \\
& \leq \left(\int |K_k(x,y)\sigma_k(y)|^p dy\right)^\frac{1}{p}\left(\int \left|\frac{u(x-y)}{\sigma_k(y)}\right|^{p'} dy\right)^\frac{1}{p'} \\
& \lesssim \sum_{|\alpha| \leq \ell} 2^{-k(n\rho/p' - |\alpha|\rho)}\left(\int |\partial_\xi^\alpha a_k(x,\xi)|^{p'} d\xi\right)^\frac{1}{p'}\left(\int \left|\frac{u(x-y)}{\sigma_k(y)}\right|^{p'} dy\right)^\frac{1}{p'} \\
& \lesssim \sum_{|\alpha| \leq \ell} 2^{-k(n\rho/p' - |\alpha|\rho)}\left(\int |\partial_\xi^\alpha a_k(x,\xi)|^{p'} d\xi\right)^\frac{1}{p'} M(u^{p'})^\frac{1}{p'}(x).
\end{aligned}
\end{equation}
Consequently, under the hypotheses of $(i)$, H\"older's inequality with exponents $p/r$ and $q/r$, the weighted boundedness of the Hardy-Littlewood maximal function and Minkowski's inequality show us that
\begin{equation} \label{5.2}
\begin{aligned}
& \|T_{a_k}(u)\|_{L^r_{\nu}(\R^n)} \\
& = \left(\int |T_{a_k}(u)(x)|^r \mu(x)^\frac{r}{p} w(x)^\frac{r}{q}dx\right)^\frac{1}{r} \\
& \lesssim \left(\int \sum_{|\alpha| \leq \ell} 2^{-kr(n\rho/p' - |\alpha|\rho)}\left(\int |\partial_\xi^\alpha a_k(x,\xi)|^{p'} d\xi\right)^\frac{r}{p'} M(u^{p'})^\frac{r}{p'}(x) \mu(x)^\frac{r}{p} w(x)^\frac{r}{q}dx\right)^\frac{1}{r} \\
& \lesssim \sum_{|\alpha| \leq \ell} 2^{-k(n\rho/p' - |\alpha|\rho)}\left(\int \left(\int |\partial_\xi^\alpha a_k(x,\xi)|^{p'} d\xi\right)^\frac{p}{p'} \mu(x)dx\right)^\frac{1}{p} \|u\|_{L^q_w(\R^n)} \\
& \lesssim \sum_{|\alpha| \leq \ell} 2^{-k(n\rho/p' - |\alpha|\rho)}\left(\int \left(\int |\partial_\xi^\alpha a_k(x,\xi)|^{p}\mu(x)dx\right)^\frac{p'}{p} d\xi \right)^\frac{1}{p'} \|u\|_{L^q_w(\R^n)} \\
& \lesssim  2^{k(m -n(\rho - 1)/p')}\|u\|_{L^q_w(\R^n)}
\end{aligned}
\end{equation}
We can then sum in $k$ to find that
\[
\|T_a(u)\|_{L^r_{\nu}(\R^n)} \leq \sum_{k=0}^\infty \|T_{a_k}(u)\|_{L^r_{\nu}(\R^n)} \lesssim 2^{k(m - n(\rho-1)/p')}\|u\|_{L^q_w(\R^n)} \lesssim \|u\|_{L^q_w(\R^n)},
\]
which completes the proof of $(i)$.

The proof of $(ii)$ is similar. When $r = 1$, $p' = q$, so we cannot use the boundedness of the Hardy-Littlewood maximal function in \eqref{5.2}. So before applying Theorem \ref{convolve} in \eqref{5.1}, we instead first take the $L^1_\nu(\R^n)$ norm of the inequality, apply H\"older's inequality as before, and then apply Young's inequality to the factor involving $\sigma_k$. When $q = \infty$, $r = p \in [2,\infty)$, so once again we cannot use the weighted boundedness of the Hardy-Littlewood maximal function, but we can use its boundedness on $L^\infty(\R^n)$.
\end{proof}

As an application of Theorem \ref{Lrsym} we now establish our main result regarding the boundedness of smooth bilinear pseudodifferential operators that fall outwith the scope of the bilinear Calder\'on-Zygmund theory. We shall prove a further result concerning a subclass of these operators in Section \ref{outwith}.
\begin{defi}\label{bilinear multiplier defn}
Given $n,N \in \N$, $m\in \mathbb{R}$ and $\rho, \delta\in [0,1]$, a symbol $a \colon \R^n \times \R^{Nn} \to \C$ belongs to the class $S^{m}_{\rho,\delta}(n,N)$ if, for each pair of multi-indices $\alpha \in \Z^{Nn}$, and $\beta \in \Z^{n}$, there exists a constant $C_{\alpha,\beta}$ such that
\begin{equation*}
|\partial_\Xi^\alpha \partial_x^\beta a(x,\Xi)| \leq C_{\alpha,\beta} \langle\Xi\rangle^{m-\rho|\alpha|+\delta|\beta|}.
\end{equation*}
In particular, a symbol $a \colon \R^n \times \R^n \times \R^{n} \to \C$ belongs to the class $S^{m}_{\rho,\delta}(n,2)$ if, for each triple of multi-indices $\alpha\in \Z^{n},$ $\beta\in \Z^{n}$ and $\gamma\in \Z^{n}$, there exists a constant $C_{\alpha, \beta, \gamma}$ such that
\begin{equation*}
|\partial_\xi^\alpha \partial_\eta^\beta \partial_{x}^{\gamma} a(x,\xi,\eta)| \leq C_{\alpha, \beta, \gamma} (1+ |\xi|^2 +|\eta|^2)^{(m-\rho(|\alpha|+|\beta|)+\delta|\gamma|)/2}.
\end{equation*}
\end{defi}

Define the adjoint operators $T^{*1}$ and $T^{*2}$ of a bilinear operator $T$ via the identities $\langle T(f,g),h\rangle_{p} = \langle f,T^{*1}(h,g)\rangle_{p} = \langle g,T^{*2}(f,h)\rangle_{p}$, where $\langle \cdot,\cdot\rangle_{p}$ denotes the dual pairing on $L^p$ ($1\leq p < \infty$). We will use Theorem 1 in \cite{BMNT} repeatedly, so we record it here.

\begin{thm} \label{1in2}
Assume that $0 \leq \delta \leq \rho \leq 1$, $\delta < 1$ and $a \in S^m_{\rho,\delta}(n,2)$. Then $T_a^{*j} = T_{b_j}$ for some $b_j \in S^m_{\rho,\delta}(n,2)$ and $j=1,2$.
\end{thm}

\begin{thm}\label{smooth pseudo bilinear}
Let $a \in S^m_{\rho,\delta}(n,2)$, with $0 \leq \delta \leq \rho \leq 1$, $\delta < 1$ and
\begin{equation} \label{stronger}
m < n(\rho - 1)\max\left\{\frac{1}{2},\left(\frac{2}{p} - \frac{1}{2}\right),\left(\frac{2}{q} - \frac{1}{2}\right),\left(\frac{3}{2} - \frac{2}{r}\right)\right\}
\end{equation}
for $p,q,r \in [1,\infty]$ $($see Figure $\ref{figure}$$)$. Then
\[
\|T_a(f,g)\|_{L^r(\R^n)} \lesssim \|f\|_{L^q(\R^n)}\|g\|_{L^p(\R^n)}
\]
for $p$,$q$ and $r$ such that $1/r = 1/q + 1/p$.
\end{thm}

\begin{figure}
\begin{center}
\input{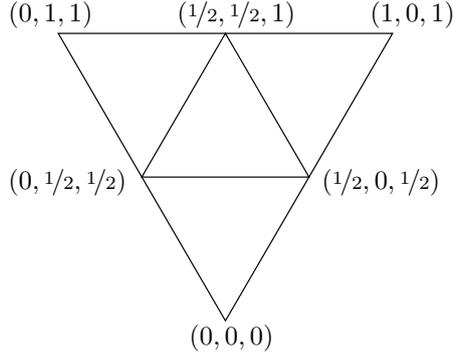}
\caption{The triangle with vertices $(0,1,1)$, $(1,0,1)$ and $(0,0,0)$ in the plane $\{(1/p,1/q,1/r) \, | \, 1/p + 1/q = 1/r\} \subset \R^3$. The right-hand side of \eqref{stronger} is linear on each of the sub-triangles depicted. It is $n(\rho - 1)/2$ at $(1/2,1/2,1)$, $(0,1/2,1/2)$ and $(1/2,0,1/2)$ and $3n(\rho - 1)/2$ at $(0,1,1)$, $(1,0,1)$ and $(0,0,0)$.} \label{figure}
\end{center}
\end{figure}
\begin{proof}
Observe that the condition $r \geq 1$ ensures that we cannot have both $p$ and $q$ less than $2$. Theorem \ref{1in2} tells us that the adjoint operators $T_a^{*1}$ and $T_a^{*1}$ are operators in the same class as $T_a$. Considering these adjoint operators if necessary, we can reduce the proof of the theorem to the special case $p,q\geq 2$.

We consider a Littlewood-Paley partition of unity $\{\varphi_{k}\}_{k\geq 0}$ defined as in \eqref{lp} with $N=1.$ We then set $a_{j,k}(x,\xi,\eta) = a(x,\xi,\eta)\varphi_j(\xi)\varphi_k(\eta)$, so $a(x,\xi,\eta) = \sum_{j,k=0}^{\infty} a_{j,k}(x,\xi,\eta)$. We can write
\begin{equation}
\begin{aligned} \label{iteration}
T_{a_{j,k}}(f,g)(x) & = \int \int a_{j,k}(x,\xi,\eta) \widehat{f}(\xi) \widehat{g}(\eta) e^{ix\cdot(\xi + \eta)} d\xi d\eta \\
& = \int \left( \int a_{j,k}(x,\xi,\eta)  \widehat{g}(\eta) e^{ix\cdot\eta} d\eta \right) \widehat{f}(\xi) e^{ix\cdot\xi} d\xi \\
& = \int A_{j,k}(g;x,\xi) \widehat{f}(\xi) e^{ix\cdot\xi} d\xi = T_{A_{j,k}(g;\cdot,\cdot)}(f)(x),
\end{aligned}
\end{equation}
where $T_{A_{j,k}(g;\cdot,\cdot)}$ is a linear operator for each fixed $g$. In order to take advantage of the smoothness of the symbol while viewing the operator as written in \eqref{iteration}, we must deal with each piece of the symbol $a_{j,k}$ depending on how the size of $\xi$ relates to the size of $\eta$. More precisely, for $k \geq j$ we have that $2^k \simeq |\eta| \gtrsim |\xi| \simeq 2^j$ on the $(\xi,\eta)$-support of $a_{j,k}$. Fixing an $\varepsilon > 0$ sufficiently small, we can conclude that
\[
|\partial_\xi^\alpha\partial_\eta^\beta\partial_x^\gamma a_{j,k}(x,\xi,\eta)| \leq C_{\alpha,\beta,\gamma}\langle2^j + 2^k\rangle^{m-\rho(|\alpha|+|\beta|)+\delta|\gamma|} \leq C_{\alpha,\beta,\gamma}2^{j(m_1-\rho|\alpha|-\varepsilon)}2^{k(m_2-\rho|\beta|+\delta|\gamma|-\varepsilon)}
\]
for $m_1,m_2 \leq 0$ such that $m_1 + m_2 = m + 2\varepsilon$. Thus, by \cite[p.~322]{S}, if $m_2 \leq n(\rho - 1)(1/p' - 1/2)$ then we have that
\[
\|\partial_\xi^\alpha A_{j,k}(g;\cdot,\xi)\|_{L^p(\R^n)} \leq C_\alpha 2^{k(m_1-\rho|\alpha|-\varepsilon)}2^{-j\varepsilon}\|g\|_{L^p(\R^n)}.
\]
This shows that $A_{j,k}(g;\cdot,\cdot) \in L^pS^{m_1}_\rho.$ Therefore, using the fact that $p\geq 2,$ an application of Theorem \ref{Lrsym} with the assumption $m_1 < n(\rho - 1)/p'$ yields
\begin{equation} \label{pieceest}
\|T_{a_{j,k}}(f,g)\|_{L^r(\R^n)} \leq C2^{-k\varepsilon}2^{-j\varepsilon}\|f\|_{L^q(\R^n)} \|g\|_{L^p(\R^n)}.
\end{equation}

For the case $k < j$, we can repeat the same argument, but reverse the roles of $\xi$ and $\eta$, and $p$ and $q$, to obtain once again \eqref{pieceest}. We then sum in $j$ and $k$ to obtain boundedness from $L^q \times L^p$ to $L^r$ provided
\begin{equation} \label{weaker}
m < 2n(\rho - 1)\max\left\{\left|\frac{1}{p} - \frac{1}{2}\right|,\left|\frac{1}{q} - \frac{1}{2}\right|\right\} + \frac{n}{2}(\rho - 1).
\end{equation}

This result can be improved by using duality and once again applying Theorem \ref{1in2}. We are concerned with triples of reciprocals of exponents $(1/p,1/q,1/r)$ such that $1/p + 1/q = 1/r$ and $p,q,r \in [1,\infty]$. The set of such triples is a closed triangle with vertices $(0,1,1)$, $(1,0,1)$ and $(0,0,0)$ in the plane $\{(1/p,1/q,1/r) \, | \, 1/p + 1/q = 1/r\}$, which itself lies in $\R^3$ (see Figure \ref{figure}). Considering the edge of the triangle with end-points $(0,1,1)$ and $(1,0,1)$ (that is, where $r=1$) we have proved that $T_a$ maps $L^p \times L^{p'}$ to $L^1$ (of course, here we require $q=p'$) for
\begin{equation*}
m < 2n(\rho - 1)\left|\frac{1}{p} - \frac{1}{2}\right| + \frac{n}{2}(\rho - 1) = 2n(\rho - 1)\left|\frac{1}{p'} - \frac{1}{2}\right| + \frac{n}{2}(\rho - 1).
\end{equation*}
This agrees with the statement of the theorem for these exponents. Equally the theorem in the case $p = q = r = \infty$ (corresponding to the point $(0,0,0)$) is also included in the condition \eqref{weaker}.

Considering the adjoint $T^{*2}$, using duality and applying Theorem \ref{1in2} allows us to conclude that $T_a$ maps $L^p \times L^{\infty}$ to $L^{p}$ for the same range on $m$. Once again, this agrees with the statement of the theorem for these exponents, but this time corresponds to triples $(1/p,1/q,1/r)$ on the line with end-points $(0,1,1)$ and $(0,0,0)$. Similarly, considering the adjoint $T^{*1}$, using duality and applying Theorem \ref{1in2} allows us to conclude that $T_a$ maps $L^\infty \times L^{p'}$ to $L^{p'}$, again, for the same range of $m$. Yet again, this agrees with the statement of the theorem for these exponents, but now corresponds to triples $(1/p,1/q,1/r)$ on the line with end-points $(1,0,1)$ and $(0,0,0)$.

Thus, we have proved the theorem on the edges of the triangle. Finally, the bilinear version of the Riesz-Thorin Interpolation Theorem (see \cite[p.~73]{G}) allows us to complete the proof on the interior of the triangle.
\end{proof}

\section{Boundedness of a subclass of H\"ormander-type bilinear pseudodifferential operators with $m\leq0$} \label{outwith}

In this section we will establish the boundedness of a subclass of bilinear pseudodifferential operators with symbols in the class $S^{0}_{\rho,0}(n,2)$ with $\rho \in (0, 1]$. With this goal in mind, the following lemma will prove to be useful.

\begin{lem} \label{notquite} Given a smooth function $a \colon \R^n \times \R^n \to \C$, define an operator $S$ by
\[
S(F)(x) = \int_{\R^n}\int_{\R^n} a(\xi,\eta)\widehat{F}(\xi,\eta) e^{ix\cdot(\xi+\eta)} d\xi d\eta,
\]
where $\widehat{F}$ denotes the Fourier transform in $\R^n \times \R^n$ $($that is, \eqref{ftN} with $N=2$$)$. The operator $S$ is bounded from $L^2(\R^n \times \R^n)$ to $L^2(\R^n)$ if and only if
\begin{align*}
A & := \left(\sup_{\zeta \in \R^n}\int_{\R^n} |a(\zeta -\eta,\eta)|^2 d\eta\right)^\frac{1}{2} < \infty.
\end{align*}
Moreover,
\[
\|S(F)\|_{L^2(\R^n)} \leq A\|F\|_{L^2(\R^n \times \R^n)}.
\]
\end{lem}

\begin{proof}
The boundedness of $S$ is equivalent to the boundedness of $S S^* \colon L^2(\R^n) \to L^2(\R^n)$, where $S^*$ is the adjoint operator of $S$:
\[
S^*(g)(x,y) = \int_{\R^n}\int_{\R^n} \overline{a(\xi,\eta)} \widehat{g}(\xi + \eta) e^{i(x\cdot\xi + y\cdot\eta)} d\xi d\eta.
\]
We can readily see that $\widehat{S^*}$ is given by
\[
\widehat{S^*}(g)(\xi,\eta) = \overline{a(\xi,\eta)} \widehat{g}(\xi + \eta),
\]
so
\begin{align*}
& S S^*(g) = \int_{\R^n}\int_{\R^n} a(\xi,\eta)\overline{a(\xi,\eta)} \widehat{g}(\xi + \eta) e^{ix\cdot(\xi+\eta)} d\xi d\eta \\
& = \int_{\R^n}\int_{\R^n} a(\zeta-\eta,\eta)\overline{a(\zeta-\eta,\eta)} \widehat{g}(\zeta) e^{ix\cdot\zeta} d\zeta d\eta
= \int_{\R^n}\left(\int_{\R^n} |a(\zeta-\eta,\eta)|^2 d\eta\right) \widehat{g}(\zeta) e^{ix\cdot\zeta} d\zeta.
\end{align*}
Thus, $SS^*$ is a multiplier and the condition of the lemma is exactly the condition required for its boundedness on $L^2(\R^n)$.
\end{proof}

\begin{thm} \label{modulationtype}
If $0< \rho \leq 1$ and $a \in S^0_{\rho,0}(n,2)$ are such that, for each multi-index $\alpha \in\Z^{n}$,
\begin{equation} \label{hypo}
\sup_{x_0,\zeta \in \R^n}\left(\int_{B_{1}(x_0)}\int_{\R^n} |\partial_x^\alpha a(x,\zeta - \eta,\eta)|^2 d\eta dx\right)^\frac{1}{2} < \infty,
\end{equation}
where $B_1(x_0)$ is the unit ball centred at $x_0$, then $T_a$ is bounded operator from $L^p(\R^n) \times L^q(\R^n)$ to $L^r(\R^n)$ for $p,q \in [2,\infty]$ and $r \in [1,2]$ such that $1/p + 1/q = 1/r$. This corresponds to $(1/p,1/q,1/r)$ contained in the closed triangle of $\mathrm{Figure}$ $\ref{figure}$ with vertices $(1/2,1/2,1)$, $(0,1/2,1/2)$ and $(1/2,0,1/2)$.
\end{thm}

\begin{proof}
First, let us observe that it suffices to prove the boundedness from $L^2 \times L^2$ to $L^1$. Indeed, assuming this and using Theorem \ref{1in2}, together with duality arguments and the multilinear Riesz-Thorin Theorem (see \cite[p.~73]{G}), yield the theorem.

Let us now suppose that the symbol $a$ has compact $x$-support, say contained in the unit ball $B = B_1(x_0)$. We follow \cite[pp.~234-5]{S} and write
\[
a(x,\xi,\eta) = \int \widehat{a}(\lambda,\xi,\eta) e^{i\lambda\cdot x} dx
\]
and so
\begin{align*}
T_a(f,g)(x) & = \iint a(x,\xi,\eta)\widehat{f}(\xi)\widehat{g}(\eta) e^{ix\cdot(\xi+\eta)} d\xi d\eta \\
& = \iiint \widehat{a}(\lambda,\xi,\eta) e^{i\lambda\cdot x} \widehat{f}(\xi)\widehat{g}(\eta) e^{ix\cdot(\xi+\eta)} d\xi d\eta d\lambda \\
& = \int T^\lambda_a(f,g)(x) d\lambda,
\end{align*}
where
\begin{align*}
T^\lambda_a(f,g)(x) & =  e^{i\lambda\cdot x} \iint \widehat{a}(\lambda,\xi,\eta) \widehat{f}(\xi)\widehat{g}(\eta) e^{ix\cdot(\xi+\eta)} d\xi d\eta = e^{i\lambda\cdot x} T_{\widehat{a}(\lambda,\cdot,\cdot)}(f,g)(x)
\end{align*}
and $T_{\widehat{a}(\lambda,\cdot,\cdot)}$ is a multiplier operator for each $\lambda$. Therefore, since
\[
\widehat{a}(\lambda,\xi,\eta) = \int_B a(x,\xi,\eta) e^{-i\lambda\cdot x} dx,
\]
we have that
\[
(i\lambda)^\alpha\widehat{a}(\lambda,\xi,\eta) = \int_B \partial_x^\alpha a(x,\xi,\eta) e^{-i\lambda\cdot x} dx
\]
and so
\[
|\lambda^\alpha|^2 \int |\widehat{a}(\lambda,\zeta-\eta,\eta)|^2 d\eta \lesssim \int\int_B |\partial_x^\alpha a(x,\zeta-\eta,\eta)|^2 dx d\eta.
\]
This means that, by our hypotheses \eqref{hypo},
\[
\int |\widehat{a}(\lambda,\zeta-\eta,\eta)|^2 d\eta \lesssim \langle \lambda \rangle^{-N}
\]
for any $N \in \Z$.
Consequently, by Lemma \ref{notquite},
\[
\|T_{\widehat{a}(\lambda,\cdot,\cdot)}(f,g)\|_{L^1(B)} \lesssim \|T_{\widehat{a}(\lambda,\cdot,\cdot)}(f,g)\|_{L^2(B)}
\lesssim \langle \lambda \rangle^{-N} \|f\|_{L^2(\R^n)}\|g\|_{L^2(\R^n)}.
\]
Once again using the support properties of $a$, we conclude that
\[
\|T_a(f,g)\|_{L^1(\R^n)} \leq \int \|T^\lambda_a(f,g)\|_{L^1(B)} d\lambda \lesssim \int \langle \lambda \rangle^{-N} \|f\|_{L^2(\R^n)}\|g\|_{L^2(\R^n)} d\lambda = C\|f\|_{L^2(\R^n)}\|g\|_{L^2(\R^n)},
\]
which proves the theorem under the extra hypothesis that $a$ has compact $x$-support.\\

To remove the hypothesis that $a$ has compact $x$-support, we follow the argument of \cite[pp.~236-7]{S}. Observe that it suffices to show that, for each $x_0 \in \R^n$,
\begin{equation} \label{stein}
\int_{B_{1}(x_0)} |T_a(f,g)(x)| dx \lesssim \left(\int_{\R^n} \frac{|f(y)|^2 dy}{(1+|y-x_0|)^N}\int_{\R^n} \frac{|g(z)|^2 dz}{(1+|z-x_0|)^N}\right)^{1/2},
\end{equation}
for all $N \geq 0.$ Indeed, choosing $N>n$, integrating \eqref{stein} in $x_0$, using the Cauchy-Schwarz inequality and interchanging the order of integration produces the estimate
\[
\|T_a(f,g)\|_{L^1(\R^n)} \lesssim \|f\|_{L^2(\R^n)}\|g\|_{L^2(\R^n)}.
\]

To prove \eqref{stein} we introduce the cut-off function $\psi \colon \R^n \to \R$, which is identically one on $B_{2}(x_0)$ and zero outside its concentric double $B_4(x_0)$. Define $f_1(x) = \psi(x) f(x)$ and $g_1(y) = \psi(y) g(y)$ and $F(x,y) = f(x)g(y) - f_1(x)g_1(y)$. Let $\varphi \colon \R^n \to \R$ be a second nonnegative cut-off function which is identically one on $B_{1}(x_0)$ and supported in $B_{2}(x_0)$. Using the cut-off function $\psi$, we can write
\begin{align} \label{parts}
T_{a}(f,g)(x) = T_a(f_1,g_1)(x) + \iint K(x,x-y,x-z)F(y,z) dy dz
\end{align}
where
\begin{equation*}\label{kernel of the bilinear op}
K(x,z_1 , z_2)=\iint a(x,\xi,\eta) e^{iz_1\cdot \xi} e^{iz_2\cdot \eta} d\xi d\eta.
\end{equation*}
Now using the cut-off function $\varphi,$ our previous boundedness result concerning bilinear operators with compact spatial support yields
\begin{align*}
\int_{B_{1}(x_0)} |T_{a}(f_1,g_1)(x)| dx &= \int_{B_{1}(x_0)} \varphi(x) |T_{ a}(f_1,g_1)(x)| dx\\ &\leq \int_{\R^n} |T_{\varphi a}(f_1,g_1)(x)| dx \lesssim \|f_1\|_{L^2(\R^n)}\|g_1\|_{L^2(\R^n)},
\end{align*}
and this is, in turn, controlled by
\[
\left(\int_{\R^n} \frac{|f(x)|^2 dx}{(1+|x-x_0|)^N}\int_{\R^n} \frac{|g(y)|^2 dy}{(1+|y-x_0|)^N}\right)^{1/2}
\]
for any $N \geq 0$, because of the support properties of $f_1$ and $g_1$. To estimate the contribution of the remaining term in \eqref{parts}, we use the kernel estimate of Corollary \ref{zlemma2}. We need to estimate
\begin{equation}\label{remainder kernel estim}
\int_{B_{1}(x_0)} \left|\int_{\R^n} \int_{\R^n} K(x,x-y,x-z)F(y,z) dydz\right| dx,
\end{equation}
but since $F$ is supported outside $B_{2}(x_0) \times B_{2}(x_0)$, Corollary \ref{zlemma2} yields that for all $N\geq 0$, \eqref{remainder kernel estim} is majorised by
\begin{align*}
& \int_{B_{1}(x_0)}\iint_{(B_2(x_0)\times B_2(x_0))^c} \frac{|F(y,z)|}{(1+|x-y|)^{N}(1+|x-z|)^{N}} dydz dx \\
& \lesssim \iint_{\R^n \times \R^n} \frac{|f(y)g(z)| + |f_1(y)g_1(z)|}{(1+|x_0-y|)^{N}(1+|x_0-z|)^{N}} dydz,
\end{align*}
and using the Cauchy-Schwarz inequality and the fact that $|f_1(y)g_1(z)|\lesssim |f(y)g(z)|$, we have
\begin{align*}
& \iint_{\R^n \times \R^n} \frac{|f(y)g(z)| + |f_1(y)g_1(z)|}{(1+|x_0-y|)^{N}(1+|x_0-z|)^{N}} dydz\\
& \lesssim \left(\int_{\R^n} \frac{|f(y)|^2 dy}{(1+|y-x_0|)^N}\int_{\R^n} \frac{dy}{(1+|y-x_0|)^N}\int_{\R^n} \frac{|g(z)|^2 dz}{(1+|z-x_0|)^N}\int_{\R^n} \frac{dz}{(1+|z-x_0|)^N}\right)^{1/2}\\
&\lesssim \left(\int_{\R^n} \frac{|f(y)|^2 dy}{(1+|y-x_0|)^N}\int_{\R^n} \frac{|g(z)|^2 dz}{(1+|z-x_0|)^N}\right)^{1/2},
\end{align*}
provided $N>n.$ This completes the proof of \eqref{stein} and with it, the theorem.
\end{proof}

\footnotemark[2]{Department of Mathematics, Oregon State University, Corvallis, Oregon 97331-4605, United States.}

\footnotemark[3]{Department of Mathematics and the Maxwell Institute of Mathematical Sciences, Heriot-Watt University, Colin Maclaurin Building, Edinburgh, EH14 4AS, United Kingdom.}

\end{document}